\documentclass{amsart}
\usepackage{amsbsy}
\usepackage{mathrsfs}
\usepackage{amsmath}
\usepackage{amsfonts}
\usepackage{amssymb}
\usepackage{graphicx}%
\newtheorem{theorem}{Theorem}[section]
\newtheorem{lemma}[theorem]{Lemma}

\theoremstyle{definition}

\theoremstyle{remark}
\newtheorem{remark}[theorem]{Remark}  
\numberwithin{equation}{section}

\usepackage{fancyhdr}
\usepackage{color}
\usepackage{fancyhdr}
\allowdisplaybreaks[4]

\newcommand{\abs}[1]{\lvert#1\rvert}

\begin{document}

	\title{The discrepancy in min-max statistics between two random matrices with finite third moments}

	
	\author{Zijun Chen}
	\address{Institute for Financial Studies, Shandong University,  Jinan,  250100, China.}
	\email{czj4096@gmail.com}

	\author{Yiming Chen$^*$}
	\address{Institute for Financial Studies, Shandong University,  Jinan,  250100, China.}
	\email{chenyiming960212@mail.sdu.edu.cn}
	
	\author{Chengfu Wei}
	\address{Institute for Financial Studies, Shandong University,  Jinan,  250100, China.}
	\email{chengfuwei1024@gmail.com}


	\date{}


	\begin{abstract}
	We propose a novel coupling inequality of the min-max type for two random matrices with finite absolute third moments, which generalizes the quantitative versions of the well-known inequalities by Gordon. Previous results have calculated the quantitative bounds for pairs of Gaussian random matrices. Through integrating the methods utilized by Chatterjee-Meckes and Reinert-Röllin in adapting Stein's method of exchangeable pairs for multivariate normal approximation, this study eliminates the Gaussian restriction on random matrices, enabling us to achieve more extensive results.
	\end{abstract}

	\maketitle

	\section{Main Result and Introduction}
	Let $X$,$X^{\prime}$ be $n \times m$ independent random matrices with mean zero and finite 
		absolute third moments, that is, $\mathbb{E}[X_{i,j}]=\mathbb{E}[X_{i,j}^{\prime}]=0$,  $\mathbb{E}[{|X_{i,j}|}^3]<\infty$ and $\mathbb{E}[{|X_{i,j}^{\prime}|}^3]<\infty$ for all $1 \leq i \leq n$ and $1 \leq j \leq m$. The goal in this paper is to quantitatively study the difference in the distribution of min-max statistics between two random matrices. %This is our main result.
	
	 This problem originated from the comparison of the expected values of the maximums of two Gaussian processes by Sudakov\cite{ref10},\cite{ref9} and Fernique\cite{ref11}. Similar comparisons of extreme value statistics involving two different Gaussian objects are important results in the field of probability theory, especially in the study of Gaussian processes. This is because it provides a tool for bounding and comparing two families of non-independent random variables. The classical Sudakov-Fernique inequality states that, let $\xi \sim N\left(0, \Sigma^{\xi}\right)$  and  $\eta \sim N\left(0, \Sigma^{\eta}\right)$  be independent  n-dimensional Gaussian vectors. Suppose that for all $1\leq i,j \leq n$, we have 
		$$\mathbb{E}[|\xi_i-\xi_j|^2] \geq \mathbb{E}[|\eta_i-\eta_j|^2, $$
	then $$\mathbb{E}[\max_{1 \leq i \leq n}\xi_i] \geq \mathbb{E}[\max_{1 \leq i \leq n}\eta_i].$$ 
	
	 The Sudakov-Fernique inequality sparked further research and development in the theory of stochastic processes. Vitale \cite{ref13} replaced the assumption that the expectation of each element is zero with the weaker condition that $ \mathbb{E}\left[\left(X_1, \ldots, X_n\right)\right]=\mathbb{E}\left[\left(Y_1, \ldots, Y_n\right)\right]$. Gordon \cite{ref14}\cite{ref15} and Kahane \cite{ref16} have also explored comparable inequalities within the broader context of higher-dimensional tensors $(X_{i_1,i_2,\cdots,i_d})$. In Gordon and Kahane's extended framework, the traditional minimum used in the classic Sudakov-Fernique inequality is substituted with expressions such as $\min\limits_{i_1} \max\limits_{i_2} \cdots X_{i_1,i_2,\cdots,i_d}$. It was discovered that this approach can be used to estimate the smallest singular value of the random matrix with independent Gaussian elements, as found in \cite{ref18}. Moreover, the majorizing measures theory developed by Fernique and Talagrand \cite{ref17}, was inspired by the Sudakov-Fernique inequality, for more details, please refer to \cite{ref3}.%{\color{red}!!!!!!!!!!}.
	
	The study on the quantitative version of the Sudakov-Fernique inequality has also attracted significant attention. Chaterjee \cite{ref12} obtained an asymptotically sharp error bound for the maximum of a finite collection of Gaussian random variables. Peccati and Turchi\cite{ref8} extend such quantitative bounds to the setting considered by Gordon \cite{ref14},\cite{ref15} of min-max statistics of Gaussian random matrices.

	\begin{theorem}[\cite{ref8}]\label{101}
	Let $X$ and $X^{\prime}$ are $n \times m$ independent Gaussian matrices with mean zero, that is,  $\mathbb{E}[X_{i,j}]=\mathbb{E}[X_{i,j}^{\prime}]=0$. For $1 \leq i_1,i_2 \leq n$ and $1 \leq j_1,j_2 \leq m$, let $\sigma=\max\limits_{i_1,j_1,i_2,j_2}|\mathbb{E}[X_{i_1,j_1}X_{i_2,j_2}-X_{i_1,j_1}^{\prime}X_{i_2,j_2}^{\prime}]|$. There exist a constant $ C > 0 $ such that for all borel subset $A \subset \mathbb{R}$ we have $$\mathbb{P}(|\min_{1 \leq i \leq n}\max_{1 \leq j \leq m}X_{i,j}-\min_{1 \leq i \leq n}\max_{1 \leq j \leq m}X_{i,j}^{\prime}|>5\epsilon ) \leq C \frac{\log nm}{\tau^2}\sigma,$$ for all $\tau >0$.
	\end{theorem}

	Inspired by the works of Chatterjee \cite{ref6}, we aim to derive a quantitative coupling inequality for the min-max type of random matrices with a stronger condition. Specifically, we only require the existence of the third-order moments for the elements of the two random matrices to be compared. Our core technique is derived from the work of Chernozhukov, Chetverikov and Kato \cite{ref1}, which utilizes exchangeable pairs and symmetry through Stein's method to obtain the main results. Now we are ready to present our main result. %Furthermore, when the conditions are the same, we are able to recover the result of [\cite{ref8}].

\begin{theorem}\label{thm1}
		Let $X$,$X^{\prime}$ be $n \times m$ independent random matrices defined before. Then for every $\beta > 0$, $\delta > 0$ and $\tau >\frac{1}{\beta(1+\delta)}$, we have 
		$$\begin{aligned}
			&\mathbb{P}(|\min_{1 \leq i \leq n}\max_{1 \leq j \leq m}X_{i,j}-\min_{1 \leq i \leq n}\max_{1 \leq j \leq m}X_{i,j}^{\prime}|>2(\frac{\log n}{\beta\delta}\vee \frac{\log m}{\beta})+3\tau)\\		
			\leq& \dfrac{\varepsilon +C\beta(1+\delta)\tau^{-1}(B_1+B_1^{\prime}+B_3+\beta(1+\delta)(B_2+B_2^{\prime}))}{1-\varepsilon}.
		\end{aligned}$$
		
		Moreover, if $X$ and $X^\prime$ are Gaussian random matrices, we have $$\begin{aligned}
			&\mathbb{P}(|\min_{1 \leq i \leq n}\max_{1 \leq j \leq m}X_{i,j}-\min_{1 \leq i \leq n}\max_{1 \leq j \leq m}X_{i,j}^{\prime}|>2(\frac{\log n}{\beta\delta}\vee \frac{\log m}{\beta})+3\tau) \\
			\leq& \dfrac{\varepsilon +C\beta(1+\delta)\tau^{-1}B_3}{1-\varepsilon}.
		\end{aligned}$$ 
		
		where $\varepsilon=\varepsilon_{\beta,\delta,\tau}$ is given by
		
		\begin{align}\label{1}
			\varepsilon=\sqrt{e^{-\alpha}(1+\alpha)}<1,\quad \alpha=\beta^2(1+\delta)^2\tau^2-1>0,
		\end{align}	
		and
		
		$$\begin{aligned}
			&B_1=\mathbb{E}[\max\limits_{i_1,j_1,i_2,j_2}|X_{i_1,j_1}X_{i_2,j_2}-\mathbb{E}[X_{i_1,j_1}X_{i_2,j_2}]|],\\
			&B_1^{\prime}=\mathbb{E}[\max\limits_{i_1,j_1,i_2,j_2}|X_{i_1,j_1}^{\prime}X_{i_2,j_2}^{\prime}-\mathbb{E}[X_{i_1,j_1}^{\prime}X_{i_2,j_2}^{\prime}]|],\\
			&B_2=\mathbb{E}[\max\limits_{i,j}|X_{i,j}|^3],\\
			&B_2^{\prime}=\mathbb{E}[\max\limits_{i,j} |X_{i,j}^{\prime}|^3],\\
			&B_3=\max\limits_{i_1,j_1,i_2,j_2}|\mathbb{E}[X_{i_1,j_1}X_{i_2,j_2}-X_{i_1,j_1}^{\prime}X_{i_2,j_2}^{\prime}]|.
		\end{aligned}$$
		
	\end{theorem}

%Inspired by Chatterjee \cite{ref6} and V. Chernozhukov, D. Chetverikov, and K. Kato \cite{ref1}, we develop a new coupling inequality for the min-max type of random matrix using exchangeable pairs and symmetry. Building on the work of Giovanni Peccati and Nicola Turchi \cite{ref8}, we employ smooth functions to approximate the measure and calculate the third-order remainder of the Taylor expansion through Stein's method. Additionally, our new theorem eliminates the need to truncate normal random vectors, thereby simplifying its application.
	\begin{remark}
			The elements in matrices are not necessarily independent, and it can be observed that if the elements in  Theorem \ref{thm1} degenerate to the Gaussian case, let $\beta=\log nm/\tau$ and $\delta=1$ we have $$\mathbb{P}	(\abs{X-X^{\prime}}>5\tau)\leq \frac{\varepsilon+C\frac{\log nm}{\tau^2}\sigma}{1-\varepsilon},$$ where $\varepsilon=\frac{2\sqrt{e}\log nm}{(nm)^{2\log{nm}}}\ll \frac{\log nm}{\tau^2}\sigma$. Obviously, it is consistent with the conclusion in Theorem \ref{101}.%\cite{ref8}.
		\end{remark}	
	\subsection{Notation}
	Throughout the paper, we will always work on a fixed probability space $(\Omega,\mathcal{F},\mathbb{P})$. We use the notation$\left\|f\right\|_{\infty}:=\sup_{t\in X}|f(t)|$. The indicator function of a subset $A$ of a set $X$ is denoted as $1_A: X \rightarrow\{0,1\}$. For a subset $A$ of semimetric space $(T,d)$, let $A^{\varepsilon}$ denote the $\varepsilon$-enlargement of $A$, that is $A^{\varepsilon}=\{x \in T:d(x,A) \leq \varepsilon\}$ where $d(x,A)=\inf_{y \in A}d(x,y)$. For $a,b \in \mathbb{R}$, we employ the notation $a \vee b=\max\{a,b\}$. Unless otherwise stated, $C>0$ denote universal constants of which the values may change from place to place.
	
	%Now, we are ready to prepare our main result.
	\section{proof of the main result}
	We divide the proof into several steps. Firstly, a version of Strassen's theorem below allows us to quantify the comparison of two random variables, then we approximate the nonsmooth map $x \mapsto 1_{A}(\min\limits_{1 \leq i \leq n} \max\limits_{1 \leq j \leq m} x_{i,j} )$ by a smooth function for any matrix $x\in \mathbb{R}^{n \times m}$,  whcih allows us to compare $\min\limits_{1 \leq i \leq n}\max\limits_{1 \leq j \leq m}X_{i,j}$ and $\min\limits_{1 \leq i \leq n}\max\limits_{1 \leq j \leq m}X^{\prime}_{i,j}$ by Stein's method.

	\begin{lemma}[Kato\cite{ref1} Lemma 4.2]\label{kk4.2}
		
		Let $\mu$ and $\nu$ be Borel probability measures on $\mathbb{R}$, and let $V$ be a random variable defined on a probability space $(\Omega,\mathcal{F},\mathbb{P})$ with distribution $\mu$. Suppose that the probability space $(\Omega,\mathcal{F},\mathbb{P})$ admits a uniform random variable on $(0, 1)$ independent of $V$. Let  $\varepsilon > 0$ and $\delta > 0$ be two positive constants. Then there exists a random variable $W$, defined on $(\Omega,\mathcal{F},\mathbb{P})$, with distribution $\nu$ such that	$\mathbb{P}	(|V - W| > \delta) \leq \varepsilon$ if and only if
		 $\mu(A) \leq \nu(A^\delta)	+ \varepsilon$ for every Borel subset	$A$ of $\mathbb{R}$.
	\end{lemma}	

	By the Lemma \ref{kk4.2}, it's clear that the conclusion follows if we can prove that for every Borel subset $A$, $$\begin{aligned}
		&\mathbb{P}(\min\limits_{1 \leq i \leq n}\max\limits_{1 \leq j \leq m}X_{i,j} \in A)\\
		 \leq&\mathbb{P}(\min\limits_{1 \leq i \leq n}\max\limits_{1 \leq j \leq m}X^{\prime} \in A^{2\lambda+3\tau})\\&+\frac{\varepsilon+C\beta(1+\delta)\tau^{-1}(B_1+B_1^{\prime}+B_3+\beta(1+\delta)(B_2+B_2^{\prime}))}{1-\varepsilon}.\\	\end{aligned}$$
	
Now we approximate the nonsmooth map by a smooth function with two step. We introduce a smooth function $F^{\beta\delta}(x):\mathbb{R}^{n \times m} \rightarrow \mathbb{R}$ to approximate the nonsmooth map $x \mapsto \min\limits_{1 \leq i \leq n }\max\limits_{1 \leq j \leq m} x_{i,j} )$ first, and then derive a appropriate estimation of indicator function.

	\begin{lemma}
		For $\beta, \delta>0$, $F^{\beta\delta}:\mathbb{R}^{n \times m} \rightarrow \mathbb{R}$ defined by
		
		$$F^{\beta \delta}(x)=-\frac{1}{\beta \delta}\log\sum_{i=1}^{n}(\sum_{j=1}^{m}\exp(\beta x_{i,j}))^{-\delta}.$$ For every $x \in \mathbb{R}^{n \times m}$
		
		\begin{align}\label{8}
		\min\limits_{1 \leq i \leq n} \max\limits_{1 \leq j \leq m} x_{i,j}-\frac{1}{\beta\delta}\log n \leq F^{\beta \delta} (x) \leq \min\limits_{1 \leq i \leq n} \max\limits_{1 \leq j \leq m} x_{i,j} + \frac{1}{\beta} \log m.
		\end{align}
		
	\end{lemma}
	\begin{remark}
		This function arises in the definition of free energy in spin glasses \cite{ref4}. From inequality (\ref{8}), we know that as $\beta$ and $\delta$ tend to $+\infty$, $F^{\beta \delta}(x)$ tends to $\min\limits_{1 \leq i \leq n} \max\limits_{1 \leq j \leq m} x_{i,j}$.
	\end{remark}
	\begin{proof}
		
		Let $d \in \mathbb{N}$ and $z \in \mathbb{R}^d$, when $\beta>0$ the following inequality holds
		$$
		\frac{1}{d} \sum_{i=1}^d e^{\beta z_i} \leq \max_{1 \leq i \leq d} e^{\beta z_i} \leq \sum_{i=1}^d e^{\beta z_i},
		$$
		in particular
		$$
		\frac{1}{\beta} \log \sum_{i=1}^d e^{\beta z_i}-\frac{\log d}{\beta} \leq \max_{1 \leq i \leq d}z_{i} \leq \frac{1}{\beta} \log \sum_{i=1}^d e^{\beta z_i} .
		$$
		
		Similarly for the minimum instead, it holds that, for every $\beta^{\prime}>0$,
		$$
		-\frac{1}{\beta^{\prime}} \log \sum_{i=1}^d e^{-\beta^{\prime} z_i} \leq \max_{1 \leq i \leq d} z_{i} \leq-\frac{1}{\beta^{\prime}} \log \sum_{i=1}^d e^{-\beta^{\prime} z_i}+\frac{\log d}{\beta^{\prime}} .
		$$
		
		This allows us to control its upper and lower bounds separately. For every $x \in \mathbb{R}^{n \times m}$
		$$\begin{aligned}
			\min_{1 \leq i \leq n}\max_{1 \leq j \leq m}x_{i,j}\geq&-\frac{1}{\beta^{\prime}}\log \sum_{i=1}^n \exp(-\beta^{\prime}\max_{1 \leq j \leq m}x_{i,j})\nonumber\\
			\geq&-\frac{1}{\beta^{\prime}}\log \sum_{i=1}^n \exp(-\beta^{\prime}\frac{1}{\beta}(\log \sum_{i=1}^n \exp(-\beta^{\prime}x_{i,j})-\log m))\nonumber\\
			=&-\frac{1}{\beta ^{\prime}}\log\sum_{i=1}^{n}\exp(-\beta^{\prime}\frac{1}{\beta}\log\sum_{j=1}^{m}\exp(\beta x_{i,j}))-\frac{1}{\beta}\log m,\nonumber
		\end{aligned}$$
		
		and
		
		\begin{align}
			\min_{1 \leq i \leq n}\max_{1 \leq j \leq m}x_{i,j}\leq&-\frac{1}{\beta^{\prime}}\log \sum_{i=1}^n \exp(-\beta^{\prime}\max_{1 \leq j \leq m}x_{i,j})+\frac{1}{\beta^{\prime}}\log n\nonumber\\
			\leq&-\frac{1}{\beta ^{\prime}}\log\sum_{i=1}^{n}\exp(-\beta^{\prime}\frac{1}{\beta}\log\sum_{j=1}^{m}\exp(\beta x_{i,j}))+\frac{1}{\beta^{\prime}}\log n.\nonumber
		\end{align}
		
		Notice that, for $\beta^{\prime}=\beta\delta$ we have
		$$-\frac{1}{\beta ^{\prime}}\log\sum_{i=1}^{n}\exp(-\beta^{\prime}\frac{1}{\beta}\log\sum_{j=1}^{m}\exp(\beta x_{i,j}))=F^{\beta\delta}(x),
		$$
		which concludes the proof.
	\end{proof}

%	\begin{remark}
%		We obtained a good approximation of c by the smooth function $F^{\beta\delta}(x)$.
%	\end{remark}
	Hence, for every Borel subset $A$ of $\mathbb{R}$ we have	
	$$\mathbb{P}(\min\limits_{1 \leq i \leq n}\max\limits_{1 \leq j \leq m}X_{i,j} \in A) \leq \mathbb{P}(F^{\beta \delta} (X) \in A^{\lambda}) = \mathbb{E}[1_{A^{\lambda}}(F^{\beta \delta} (X) )].$$
	We introduce the following lemma to approximate the indicator function.
	\begin{lemma}[Kato\cite{ref1} Lemma 4.2]\label{lem1}
		
	 Let  $\phi=\beta(1+\delta)>0$  and  $\tau>1 / \phi$ . For every Borel subset  $Q$  of  $\mathbb{R}$ , there exists a smooth function  $g: \mathbb{R} \rightarrow \mathbb{R}$  such that  $\left\|g^{\prime}\right\|_{\infty} \leq \tau^{-1}$, $\left\|g^{\prime \prime}\right\|_{\infty} \leq   C \phi \tau^{-1}$, $\left\|g^{\prime \prime \prime}\right\|_{\infty} \leq C \phi^{2} \tau^{-1}$ , and
		
		\begin{align}\label{4}
			(1-\varepsilon) 1_{Q} (t)\leq g(t) \leq \varepsilon+(1-\varepsilon) 1_{Q^{3\tau}}(t) \quad \forall t \in \mathbb{R},
		\end{align}

		where  $\varepsilon=\varepsilon_{ \beta, \delta , \tau}$  is given by (\ref{1}).
	\end{lemma}
	
	Apply Lemma \ref{lem1} to  $Q=A^{\lambda}$  to construct a suitable function  $g$. Then
	
	\begin{align}\label{10}
		\mathbb{E}\left[1_{A^{\lambda}}(F^{\beta\delta}(X) )\right] \leq(1-\varepsilon)^{-1} \mathbb{E}\left[g \circ F^{\beta\delta}(X)\right] .
	\end{align}

	 Then we ought to discuss the properties of the function $g \circ F^{\beta\delta}$, so that we can bound the right term of the inequality of (\ref{10}).%{\color{red}!!!!!!!!!!}
	
	\begin{lemma}\label{lem3} Let $\beta>0$ and $\delta>0$. For every thrice continuously differentiable function $g$, we have
		\begin{align}\label{2}\sum_{i_1,i_2=1}^{n}\sum_{j_1,j_2=1}^{m}\frac{\partial^2(g \circ F^{\beta\delta})(x)}{\partial x_{i_1,j_1}\partial x_{i_2,j_2}} \leq& \left\|g^{\prime \prime}\right\|_{\infty}+2\beta(1+\delta)\left\|g^{\prime}\right\|_{\infty},\\
		\label{3}\sum_{i_1,i_2,i_3=1}^{n}\sum_{j_1,j_2,j_3=1}^{m}\frac{\partial^3(g \circ F^{\beta\delta})(x)}{\partial x_{i_1,j_1}\partial x_{i_2,j_2}\partial x_{i_3,j_3}} \leq& \left\|g^{\prime \prime \prime}\right\|_{\infty}+6\beta(1+\delta)\left\|g^{\prime \prime}\right\|_{\infty}+6\beta^2(1+\delta)^2\left\|g^{ \prime}\right\|_{\infty}.
		\end{align}
		
	\end{lemma}
	
	\begin{proof}
		
		It's convenient to define the functions, for any $1 \leq i \leq n$ and $1 \leq j \leq m$ , $p_{i,j} , q_i : \mathbb{R}^{n \times m} \rightarrow \mathbb{R}$ by
		\begin{align}
			p_{i,j}(x):=\frac{\exp(\beta x_{i,j})}{\sum\limits_{j=1}^{m}\exp(\beta x_{i,j})},
		\end{align}

		$$q_i(x):=\dfrac{(\sum\limits_{j=1}^{m}\exp(\beta x_{i,j}))^{-\delta}}{\sum\limits_{i=1}^{n}(\sum\limits_{j=1}^{m}\exp(\beta x_{i,j}))^{-\delta}}.$$
		
		Let $\delta_{i,j}=\begin{cases}
			1 & \text{ if } i=j\\
			0 & \text{ if } i \ne j\\
		\end{cases}$. A direct computation shows that 
		
		$$\frac{\partial F^{\beta\delta}(x)}{\partial x_{i_1,j_1}}=\pi_{i_1,j_1}(x),$$
		
		$$\frac{\partial^2 F^{\beta\delta}(x)}{\partial x_{i_1,j_1}\partial x_{i_2,j_2}}=\beta\omega_{i_1,j_1,i_2,j_2}(x),$$
		$$\frac{\partial^3 F^{\beta\delta}(x)}{\partial x_{i_1,j_1}\partial x_{i_2,j_2}\partial x_{i_3,j_3}}=\beta^2\gamma_{i_1,j_1,i_2,j_2,i_3,j_3}(x),$$
		
		where
		
		$$\begin{aligned}
			\pi_{i_1,j_1}(x) =& p_{i_1,j_1}(x)q_{i_1}(x),\\
			\omega_{i_1,j_1,i_2,j_2}\left(x\right) =& \pi_{i_1,j_1}\left(x\right)\left(\delta \pi_{i_2,j_2}\left(x \right)+\delta_{i_1,i_2}\left(\delta_{j_1,j_2}-\left(1+\delta \right) p_{i_2,j_2}\left(x\right)\right)\right),\\
			\gamma_{i_1,j_1,i_2,j_2,i_3,j_3}(x) = &\omega_{i_1,j_1,i_2,j_2}\left(x\right)\left(\delta \pi_{i_3,j_3}\left(x\right)+\delta_{i_1,i_3}\left(\delta_{j_1,j_3}-\left(1+\delta \right)p_{i_3,j_3}\left(x\right)\right)\right)\\
				&+\delta \pi_{i_1,j_1}\left(x\right)\omega_{i_2,j_2,i_3,j_3}\left(x\right)\\
				&-\delta_{i_1,i_2}\delta_{i_2,i_3}(1+\delta)\pi_{i_1,j_1}(x)p_{i_2,j_2}(x)(\delta_{j_2,j_3}-p_{i_3,j_3}(x)).\\
		\end{aligned}$$
	
		Consider for every $x \in \mathbb{R}^{n \times m}$
		$$p_{i,j}(x) \geq 0, \quad q_{i}(x) \geq 0,\quad \sum_{i=1}^{n} q_i(x)=1, $$
		and for every $1 \leq i \leq n$ 
		$$\sum_{j=1}^{m}p_{i,j}(x)=1 .$$
		
		By these conditions, we have
		$$
			\pi_{i,j}(x) \geq 0,\quad \sum_{i=1}^{n}\sum_{j=1}^{m}\pi_{i,j}(x)=1,$$ $$\begin{aligned}
			\sum_{i_1,i_2=1}^{n}\sum_{j_1,j_2=1}^{m}|\omega_{i_1,j_1,i_2,j_2}(x)|=&\sum_{i_1,i_2=1}^{n}\sum_{j_1,j_2=1}^{m}|\pi_{i_1,j_1}(x)[\delta\pi_{i_2,j_2}(x)+\delta_{i_1,i_2}(\delta_{j_1,j_2}-(1+\delta)p_{i_2,j_2}(x))]|\\ \leq&\sum_{i_1,i_2=1}^{n}\sum_{j_1,j_2=1}^{m}\pi_{i_1,j_1}(x)[\delta \pi_{i_2,j_2}(x)+\delta_{i_1,i_2}(\delta_{j_1,j_2}+(1+\delta)p_{i_2,j_2}(x))]\\
			\leq& 2(1+\delta),			
		\end{aligned}$$
		similarly,
		$$\sum_{i_1,i_2,i_3=1}^{n}\sum_{j_1,j_2,j_3=1}^{m}|\gamma_{i_1,j_1,i_2,j_2,i_3,j_3}(x)| \leq 4(1+\delta)^2+2\delta(1+\delta)+2(1+\delta) = 6(1+\delta)^2.
		$$

		Inequalities (\ref{2}) and (\ref{3})  follow from these relations and the following computation :
		$$\begin{aligned}
			\frac{\partial (g \circ F^{\beta\delta})(x)}{\partial x_{i_1,j_1}}= & (g^{\prime} \circ F_{\beta})(x) \pi_{i_1,j_1}(x), \\
			\frac{\partial^2 (g \circ F^{\beta\delta})(x)}{\partial x_{i_1,j_1}\partial x_{i_2,j_2}}= & (g^{\prime \prime} \circ F_{\beta})(x) \pi_{i_1,j_1}(x) \pi_{i_2,j_2}(x)+(g^{\prime} \circ F_{\beta})(x) \beta \omega_{i_1,j_1,i_2,j_2}(x), \\
			\frac{\partial^3 (g \circ F^{\beta\delta})(x)}{\partial x_{i_1,j_1}\partial x_{i_2,j_2}\partial x_{i_3,j_3}}= & (g^{\prime \prime \prime} \circ F_{\beta})(x) \pi_{i_1,j_1}(x) \pi_{i_2,j_2}(x)\pi_{i_3,j_3}(x)\\
			& +(g^{\prime \prime} \circ F_{\beta})(x) \beta(\omega_{i_1,j_1,i_2,j_2}(x)\pi_{i_3,j_3}(x)+\omega_{i_1,j_1,i_3,j_3}(x)\pi_{i_2,j_2}(x) \\
			& +\omega_{i_2,j_2,i_3,j_3}(x)\pi_{i_1,j_1}(x)) \\
			& +(g^{\prime} \circ F_{\beta})(x) \beta^{2} \gamma_{i_1,j_1,i_2,j_2,i_3,j_3}(x) .
		\end{aligned}$$ 
	\end{proof}
	
	In order to compare $\mathbb{E}[f(X)]$ and $\mathbb{E}[f(X^{\prime})]$ where $f=g \circ F^{\beta\delta}$. We construct two Gaussian random matrices $Y$ and $Y^{\prime}$ in $\mathbb{R}^{n \times m}$ with mean zero and the same covariance, i.e. $\mathbb{E}[Y_{i_1,j_1}]=\mathbb{E}[Y_{i_1,j_1}^{\prime}]=0$,  $\mathbb{E}[X_{i_1,j_1}X_{i_2,j_2}]=\mathbb{E}[Y_{i_1,j_1}Y_{i_2,j_2}]$ and $\mathbb{E}[X_{i_1,j_1}^{\prime}X_{i_2,j_2}^{\prime}]=\mathbb{E}[Y_{i_1,j_1}^{\prime}Y_{i_2,j_2}^{\prime}]$ for all $1 \leq i_1,i_2 \leq n$ and $1 \leq j_1,j_2 \leq m$. Therefore, through the triangle inequality,  we can approximate $f(X)-f(Y)$, $f(Y)-f(Y^{\prime})$ and $f(Y^{\prime})-f(X^{\prime})$ by Stein's method easily. The Lemmas \ref{lem2} and \ref{lem333} will enable us to control the $f(X)-f(Y)$, $f(Y^{\prime})-f(X^{\prime})$ and $f(Y)-f(Y^{\prime})$ separately.
	
\begin{lemma}(Elizabeth \cite{ref2})\label{lem2}
		Consider
		$$
		h(x)=\int_0^1 \frac{1}{2 t} \mathbb{E}\left[f\left(\sqrt{t} x+\sqrt{1-t} Y\right)-f\left(Y\right)\right] d t .
		$$

		Then we have for every smooth function $f$
		$$
		f(x)-\mathbb{E}\left[f\left(Y\right)\right]=\sum_{i=1}^{n}\sum_{j=1}^{m}x_{i,j} \frac{\partial h(x)}{\partial x_{i,j}}-\sum_{i_1,i_2=1}^{n}\sum_{j_1,j_2=1}^{m}x_{i_1,j_1}x_{i_2,j_2}\frac{\partial^2h(x)}{\partial x_{i_1,j_1}\partial x_{i_2,j_2}},
		$$
		and especially
		$$
		\mathbb{E}\left[f(X)-f\left(Y\right)\right]=\mathbb{E}\left[\sum_{i=1}^{n}\sum_{j=1}^{m}X_{i,j} \frac{\partial h(X)}{\partial x_{i_1,j_1}}-\sum_{i_1,i_2=1}^{n}\sum_{j_1,j_2=1}^{m}X_{i_1,j_1}X_{i_2,j_2}\frac{\partial^2h(X)}{\partial x_{i_1,j_1}\partial x_{i_2,j_2}}\right].
		$$
	\end{lemma}

\begin{lemma}[Van Handel \cite{ref3}, Vershynin \cite{ref7}]\label{lem333}
		 Let  $\xi \sim N\left(0, \Sigma^{\xi}\right)$  and  $\eta \sim N\left(0, \Sigma^{\eta}\right)$  be independent  n -dimensional Gaussian vectors, and define
		
		$$\zeta(t)=\sqrt{t}\xi +\sqrt{1-t} \eta, \quad t \in[0,1] .$$

		Then we have for every smooth function  $p$ 
		$$\frac{d}{d t} \mathbb{E}[p(\zeta(t))]=\frac{1}{2} \sum_{i, j=1}^{n}\left(\Sigma_{i j}^{\xi}-\Sigma_{i j}^{\eta}\right) \mathbb{E}\left[\frac{\partial^{2} p}{\partial x_{i} \partial x_{j}}(\zeta(t))\right] ,$$
		which means $$\begin{aligned} 
			\mathbb{E}[p(\xi)]-\mathbb{E}[p(\eta)]=&\mathbb{E}[p(\zeta(1))]-\mathbb{E}[p(\zeta(0))] \\
			=&\int_{0}^{1}\frac{1}{2}\sum_{i,j=1}^{n}(\Sigma_{ij}^{\xi}-\Sigma_{ij}^{\eta})\mathbb{E}\left[\frac{\partial^2 p}{\partial x_i \partial x_j}(\zeta (t)) \right]dt.\\  
		\end{aligned} $$
		
	\end{lemma}

%We state a new theorem that plays a key role in our proof.
	
	\begin{theorem}\label{lem4}
		
		Suppose that assumptions are as above. We state that $$|\mathbb{E}[f(X)-f(X^{\prime})]|\leq\frac{C\beta(1+\delta)}{\tau}(B_1+B_1^{\prime}+B_3+\beta(1+\delta)(B_2+B_2^{\prime})).$$
	\end{theorem}

	\begin{proof}[Proof of Theorem \ref{lem4}]
	
	We approximate $|\mathbb{E}[f(X)-f(Y)]|$ and $|\mathbb{E}[f(X^{\prime})-f(Y^{\prime})]|$ priority. Let $\bar{X}$ be an independent copy of $X$, write $\Delta_{i,j}:=\bar{X}_{i,j}-X_{i,j}$.
	
	By considering independence, we have:
	$$\mathbb{E}[\Delta_{i,j}|X]=\mathbb{E}[\bar{X}_{i,j}-X_{i,j}|X]=-X_{i,j}.$$
	Similarly, 
	$$ \begin{aligned}
		\mathbb{E}[\Delta_{i_1,j_1}\Delta_{i_2,j_2}|X]=&\mathbb{E}[(\bar{X}_{i_1,j_1}-X_{i_1,j_1})(\bar{X}_{i_2,j_2}-X_{i_2,j_2})|X]\\
		=&\mathbb{E}[X_{i_1,j_1}X_{i_2,j_2}]+X_{i_1,j_1}X_{i_2,j_2}.
 \end{aligned}$$
	
	Let
	$$R_3:=h(\bar{X})-h(X)-\sum_{i=1}^{n}\sum_{j=1}^{m}\Delta_{i,j}\frac{\partial h(X)}{\partial x_{i,j}}-\frac{1}{2}\sum_{i_1,i_2=1}^{n}\sum_{j_1,j_2=1}^{m}\Delta_{i_1,j_1}\Delta_{i_2,j_2}\frac{\partial^2h(X)}{\partial x_{i_1,j_1}\partial x_{i_2,j_2}}.$$
	
	Then we obtain that
$$	
\begin{aligned}
		0&=\mathbb{E}[h(\bar{X})-h(X)]\\
		&=\mathbb{E}[R_3+\sum_{i=1}^{n}\sum_{j=1}^{m}\Delta_{i,j}\frac{\partial h(X)}{\partial x_{i,j}}+\frac{1}{2}\sum_{i_1,i_2=1}^{n}\sum_{j_1,j_2=1}^{m}\Delta_{i_1,j_1}\Delta_{i_2,j_2}\frac{\partial^2h(X)}{\partial x_{i_1,j_1}\partial x_{i_2,j_2}}]\\
		&=\mathbb{E}[R_3+\sum_{i=1}^{n}\sum_{j=1}^{m}\mathbb{E}[\Delta_{i,j}|X]\frac{\partial h(X)}{\partial x_{i,j}}+\frac{1}{2}\sum_{i_1,i_2=1}^{n}\sum_{j_1,j_2=1}^{m}\mathbb{E}[\Delta_{i_1,j_1}\Delta_{i_2,j_2}|X]\frac{\partial^2h(X)}{\partial x_{i_1,j_1}\partial x_{i_2,j_2}}]\\
		&=\mathbb{E}[R_3-\sum_{i=1}^{n}\sum_{j=1}^{m} X_{i,j} \frac{\partial h(X)}{\partial x_{i,j}}+\frac{1}{2}\sum_{i_1,i_2=1}^{n}\sum_{j_1,j_2=1}^{m}(\mathbb{E}[X_{i_1,j_1}X_{i_2,j_2}]+X_{i_1,j_1}X_{i_2,j_2})\frac{\partial^2h(X)}{\partial x_{i_1,j_1}\partial x_{i_2,j_2}}]\\
		&=\mathbb{E}[R_3+f(Y)-f(X)+\frac{1}{2}\sum_{i_1,i_2=1}^{n}\sum_{j_1,j_2=1}^{m}(\mathbb{E}[X_{i_1,j_1}X_{i_2,j_2}]-X_{i_1,j_1}X_{i_2,j_2})\frac{\partial^2h(X)}{\partial x_{i_1,j_1}\partial x_{i_2,j_2}}],\\
	\end{aligned}
$$
	where the last equation is by the Lemma \ref{lem2}, thus,
	$$\mathbb{E}[f(X)-f(Y)]=\mathbb{E}[R_3+\frac{1}{2}\sum_{i_1,i_2=1}^{n}\sum_{j_1,j_2=1}^{m}(\mathbb{E}[X_{i_1,j_1}X_{i_2,j_2}]-X_{i_1,j_1}X_{i_2,j_2})\frac{\partial^2h(X)}{\partial x_{i_1,j_1}\partial x_{i_2,j_2}}].$$
	
	Using Lemma \ref{lem3}, we derive that
	$$\begin{aligned}
		&|\mathbb{E}[\sum_{i_1,i_2=1}^{n}\sum_{j_1,j_2=1}^{m}(\mathbb{E}[X_{i_1,j_1}X_{i_2,j_2}]-X_{i_1,j_1}X_{i_2,j_2})\frac{\partial^2h(X)}{\partial x_{i_1,j_1}\partial x_{i_2,j_2}}]|\\
		\leq&\mathbb{E}[\max\limits_{i_1,j_1,i_2,j_2}|X_{i_1,j_1}X_{i_2,j_2}-\mathbb{E}[X_{i_1,j_1}X_{i_2,j_2}]|]\sum_{i_1,i_2=1}^{n}\sum_{j_1,j_2=1}^{m}|\frac{\partial^2h(X)}{\partial x_{i_1,j_1}\partial x_{i_2,j_2}}|\quad \text{(by (\ref{2}))}\\
		\leq& \frac{C\beta(1+\delta)}{\tau}B_1.
	\end{aligned}$$
	
	Let $\theta$ be a random variable distributed uniformly on the interval $(0,1)$, independent of all other variables. Consider $R_3$ as the error in the Taylor approximation. Then one has
	
	$$\begin{aligned}
		|\mathbb{E}[R_3]|=&|\mathbb{E}[\frac{1}{2}\sum_{i_1,i_2,i_3=1}^{n}\sum_{j_1,j_2,j_3=1}^{m}\Delta_{i_1,j_1}\Delta_{i_2,j_2}\Delta_{i_3,j_3}(1-\theta)^2\frac{\partial^3 h}{\partial x_{i_1,j_1}\partial x_{i_2,j_2}\partial x_{i_3,j_3}}((1-\theta)X+\theta \bar{X})]|\\
		\leq&\frac{1}{2}\mathbb{E}[|\sum_{i_1,i_2,i_3=1}^{n}\sum_{j_1,j_2,j_3=1}^{m}\Delta_{i_1,j_1}\Delta_{i_2,j_2}\Delta_{i_3,j_3}\frac{\partial^3 h}{\partial x_{i_1,j_1}\partial x_{i_2,j_2}\partial x_{i_3,j_3}}((1-\theta)X+\theta \bar{X})|]\\
		\leq&\frac{1}{2}\mathbb{E}[\max\limits_{i_1,i_2,i_3,j_1,j_2,j_3}|\Delta_{i_1,j_1}\Delta_{i_2,j_2}\Delta_{i_3,j_3}|]\times \max_{t}\sum_{i_1,i_2,i_3=1}^{n}\sum_{j_1,j_2,j_3=1}^{m}|\frac{\partial^3 h}{\partial x_{i_1,j_1}\partial x_{i_2,j_2}\partial x_{i_3,j_3}}((1-t)X+t \bar{X})|\\
		\leq&\frac{C\beta^2(1+\delta)^2}{\tau}\mathbb{E}[\max\limits_{i_1,i_2,i_3,j_1,j_2,j_3}|\Delta_{i_1,j_1}\Delta_{i_2,j_2}\Delta_{i_3,j_3}|]\quad \text{(by (\ref{3}))}\\
		\leq&\frac{C\beta^2(1+\delta)^2}{\tau}\mathbb{E}[\max\limits_{i,j}|\Delta_{ij}^3|] \\
		\leq&\frac{C\beta^2(1+\delta)^2}{\tau}\mathbb{E}[\max\limits_{i,j}|X_{ij}^3|] \quad \text{(by symmetry)}\\
		\leq&\frac{C\beta^2(1+\delta)^2}{\tau}B_2.\\
	\end{aligned}$$ 
	
	We get
	\begin{align}\label{5}
			|\mathbb{E}[f(X)-f(Y)]|\leq \frac{C\beta(1+\delta)}{\tau}(B_1+\beta(1+\delta)B_2),
	\end{align}
	and
	\begin{align}\label{6}
	|\mathbb{E}[f(X^{\prime})-f(Y^{\prime})]|\leq \frac{C\beta(1+\delta)}{\tau}(B_1^{\prime}+\beta(1+\delta)B_2^{\prime}).
		\end{align}
	
	As for the approximation of $|\mathbb{E}[f(Y)-f(Y^{\prime})]|$, we introduce the following procedure. Consider $Y$ and $Y^{\prime}$ are Gaussian matrices, by the lemma \ref{lem333}, we have
	
\begin{align}\label{7}
		&|\mathbb{E}[f(Y)-f(Y^{\prime})]|\\
		=&\lvert\frac{1}{2}\int_{0}^{1}\sum_{i_1,i_2=1}^{n}\sum_{j_1,j_2=1}^{m}\mathbb{E}[Y_{i_1,j_1}Y_{i_2,j_2}-Y_{i_1,j_1}^{\prime}Y_{i_2,j_2}^{\prime}]\mathbb{E}\left[\frac{\partial^2f(\sqrt{t} Y+\sqrt{1-t}Y^\prime)}{\partial x_{i_1,j_1}\partial x_{i_2,j_2}}dt\right]\rvert\nonumber\\
		\leq&\max\limits_{i_1,i_2,j_1,j_2}|\mathbb{E}[Y_{i_1,j_1}Y_{i_2,j_2}-Y_{i_1,j_1}^{\prime}Y_{i_2,j_2}^{\prime}]|\max_{t}\sum_{i_1,i_2=1}^{n}\sum_{j_1,j_2=1}^{m}|\frac{\partial^2f(\sqrt{t} Y+\sqrt{1-t}Y^\prime)}{\partial x_{i_1,j_1}\partial x_{i_2,j_2}}|_{}\nonumber\\
		\leq&\frac{C\beta(1+\delta)}{\tau}\max\limits_{i_1,i_2,j_1,j_2}|\mathbb{E}[X_{i_1,j_1}X_{i_2,j_2}-X_{i_1,j_1}^{\prime}X_{i_2,j_2}^{\prime}]|\quad \text{(by (\ref{3}))}\nonumber\\
		\leq&\frac{C\beta(1+\delta)}{\tau}B_3.\nonumber
	\end{align}
	By applying the triangle inequality to inequality (\ref{5}) ,  (\ref{6}) and (\ref{7}), we can obtain
	\begin{align}\label{9}
		&\\
		|\mathbb{E}[f(X)-f(X^{\prime})]|\leq&|\mathbb{E}[f(X)-f(Y)]|+|\mathbb{E}[f(Y)-f(Y^{\prime})]|+|\mathbb{E}[f(Y^{\prime})-f(X^{\prime})]|\nonumber\\ \leq&\frac{C\beta(1+\delta)}{\tau}(B_1+B_1^{\prime}+B_3+\beta(1+\delta)(B_2+B_2^{\prime})).\nonumber
	\end{align}

	\end{proof}
	
	\begin{proof}[Proof of Theorem \ref{thm1}]

	Combine (\ref{8}), (\ref{4}), (\ref{10}) and (\ref{9}), one has
	
	$$\begin{aligned}
		&\mathbb{P}(\min\limits_{1 \leq i \leq n}\max\limits_{1 \leq j \leq m}X_{i,j} \in A)\\
		\leq&(1-\varepsilon)^{-1} \mathbb{E}[g \circ F^{\beta\delta}(X)]\\
		\leq&\frac{\mathbb{E}[g \circ F^{\beta\delta}(X^{\prime})]+C\beta(1+\delta)\tau^{-1}(B_1+B_1^{\prime}+B_3+\beta(1+\delta)(B_2+B_2^{\prime}))}{1-\varepsilon}\\
		\leq&\mathbb{P}(g \circ F^{\beta\delta}(X^{\prime}) \in A^{\lambda+3\tau})+\frac{\varepsilon+C\beta(1+\delta)\tau^{-1}(B_1+B_1^{\prime}+B_3+\beta(1+\delta)(B_2+B_2^{\prime}))}{1-\varepsilon}\\
		\leq&\mathbb{P}(\min\limits_{1 \leq i \leq n}\max\limits_{1 \leq j \leq m}X^{\prime} \in A^{2\lambda+3\tau})+\frac{\varepsilon+C\beta(1+\delta)\tau^{-1}(B_1+B_1^{\prime}+B_3+\beta(1+\delta)(B_2+B_2^{\prime}))}{1-\varepsilon}.\\	\end{aligned}$$
		 By  Lemma \ref{kk4.2}, we complete our proof.
	\end{proof}
\section{discussion}
	In this paper, we have obtained a new inequality between two random matrices for which the third moments exist, which can be regarded as a generalization of the quantitative version of the Gordon inequality. In future work, we will attempt to quantify the distribution bounds of minmax of two random matrices by assuming only the existence of the third-order moment for one random matrix $X$, and the other random matrix $F$ satisfies
$$
\sum_{q \geq 1} q q!\left\|f_q\right\|_{\mathfrak{S}^{\otimes q}}^2<\infty,
$$
where $f_q$ is an element of the symmetric $q$ th tensor product $\mathfrak{H}^{\odot q}$ (which is uniquely determined by $F)$,  i.e. the upper bound for $\mathbb{P}(\min \max F \leq x)-\mathbb{P}(\min \max X \leq x)$ where $x \in \mathbb{R}$.

\end{document}